\documentclass{amsart}
\usepackage{amsfonts,amssymb,amsmath,amsthm}
\usepackage{url}
\usepackage{enumerate}

\newtheorem{thrm}{Theorem}[section]
\newtheorem{lem}[thrm]{Lemma}
\newtheorem{prop}[thrm]{Proposition}

\theoremstyle{definition}
\newtheorem{definition}[thrm]{Definition}
\newtheorem{ddef}[thrm]{}
\newtheorem{remark}[thrm]{Remark}
\numberwithin{equation}{section}

\author{Zhichao Liu}
\address{
Department of Mathematics\\
Hebei Normal University\\
Shijiazhuang, 050024, China}
\email{lzc.12@outlook.com}
\begin{document}

\title[Injectivity of the connecting homomorphisms]{Injectivity of the connecting homomorphisms in inductive limits of Elliott-Thomsen algebras}

\begin{abstract}
Let $A$ be the inductive limit of a sequence
 $$
 A_1\,\xrightarrow{\phi_{1,2}}\,A_2\,\xrightarrow{\phi_{2,3}}\,A_3\rightarrow\cdots
 $$
 with $A_n=\bigoplus_{i=1}^{n_i}A_{[n,i]}$, where all the $A_{[n,i]}$ are Elliott-Thomsen algebras and $\phi_{n,n+1}$ are homomorphisms. In this paper, we will prove that $A$ can be written as another inductive limit
 $$
 B_1\,\xrightarrow{\psi_{1,2}}\,B_2\,\xrightarrow{\psi_{2,3}}\,B_3\rightarrow\cdots
 $$
 with $B_n=\bigoplus_{i=1}^{n_i'}B_{[n,i]'}$, where all the $B_{[n,i]'}$ are Elliott-Thomsen algebras and with the extra condition that all the   $\psi_{n,n+1}$ are injective.
\end{abstract}
\maketitle

\section{Introduction} \label{sect1}
In 1997, Li proved the result that if $A=\underrightarrow{lim}(A_n,\phi_{m,n})$ is an  inductive limit $\mathrm{C}^*$-algebra with $A_n=\bigoplus_{i=1}^{n_i}M_{[n,i]}(C(X_{[n,i]}))$, where all $X_{[n,i]}$ are graphs, $n_i$ and $[n,i]$ are positive integers,
 then one can write $A=\underrightarrow{lim}(B_n,\psi_{m,n})$, where $B_n=\bigoplus_{i=1}^{n_i'}M_{[n,i]'}(C(Y_{[n,i]'}))$ are finite direct
  sums of matrix algebras over graphs $Y_{[n,i]'}$ with the extra property that the homomorphisms $\psi_{m,n}$ are injective \cite{Li1}.
   This played an important role in the classification of simple $AH$ algebras with one-dimensional local spectra
  (see \cite{Ell1,Ell2,Li1,Li2,Li3}). This result was extended to the case of $AH$ algebras \cite{EGL},
  in which the space $X_{[n,i]}$ are replaced by connected finite simplicial complexes.

 In this article, we consider the $\mathrm{C}^*$-algebra $A$ which can be expressed as the inductive limit of a sequence
 $$
 A_1\,\xrightarrow{\phi_{1,2}}\,A_2\,\xrightarrow{\phi_{2,3}}\,A_3\rightarrow\cdots,
 $$
 where all $A_i$ are Elliott-Thomsen algebras and $\phi_{n,n+1}$ are homomorphisms.
 These algebras were introduced by Elliott in \cite{Ell3} and Thomsen in \cite{ET}, and are also called one-dimensional non-commutative finite CW complexes. We will prove that $A$ can be written as inductive limits of sequences of Elliott-Thomsen algebras with the property that all connecting homomorphisms are injective. The results in this paper will be used in \cite{ALZ} to classify real rank zero inductive limits of one-dimensional non-commutative finite CW complexes.

\section{Preliminaries}
\begin{definition}
  Let $F_1$ and $F_2$ be two finite dimensional $C^*$-algebras. Suppose that there are two  homomorphisms $\varphi_0,\varphi_1\,:F_1\rightarrow F_2$. Consider the $C^*$-algebra
  $$
  A=A(F_1,F_2,\varphi_0,\varphi_1)=\{(f,a)\in C([0,1],F_2)\oplus F_1 :f(0)=\varphi_0(a),\,\,f(1)=\varphi_1(a)\}.
  $$
\end{definition}
 These $C^*$-algebras have been introduced into the Elliott program by Elliott and Thomsen in \cite{ET}. Denote by $\mathcal{C}$ the class of all unital
 $C^*$-algebras of the form $A(F_1,F_2,\varphi_0,\varphi_1)$. (This class includes the finite dimensional $C^*$-algebras, the case $F_2=0$.) These $C^*$-algebras will be called Elliott-Thomsen algebras. Following \cite{GLN}, let us say that a unital ${\mathrm C}^*$-algebra $A\in\mathcal{C}$ is minimal, if it is indecomposable, i.e., not the direct sum of two or more ${\mathrm C}^*$-algebras in $\mathcal{C}$.

\begin{prop}[\cite{GLN}]
  Let $A=A(F_1,F_2,\varphi_0,\varphi_1)$, where $F_1=\bigoplus_{j=1}^p M_{k_j}(\mathbb{C})$, $F_2$ $=\bigoplus_{i=1}^lM_{l_i}(\mathbb{C})$ and $\varphi_0,\varphi_1 :F_1\rightarrow F_2$ be two homomorphisms.
  Let $\varphi_{0*},\varphi_{1*}:K_0(F_1)=\mathbb{Z}^p\rightarrow K_0(F_1)=\mathbb{Z}^l$ be represented by matrices $\alpha=({\alpha_{ij}})_{l\times p}$ and $\beta=({\beta_{ij}})_{l\times p}$, where $\alpha_{ij},\beta_{ij}\in \mathbb{Z}_+$ for each pair $i,j$. Then
  $$
  K_0(A)=Ker(\alpha -\beta),
%  \left\{
%  \left(
%\begin{array}{c}
%  x_1 \\
%  x_2 \\
%   \vdots\\
%  x_p
%\end{array}
%  \right)
%\in \mathbb{Z}^p,\,\,with\,\,(\alpha -\beta)
% \left(
%\begin{array}{c}
%  x_1 \\
 % x_2 \\
%   \vdots\\
%  x_p
%\end{array}
%  \right)
%  =0
%\right\}
% $$
% $$
\quad
 K_1(A)=\mathbb{Z}^l/Im(\alpha-\beta).
 $$
\end{prop}
\begin{ddef}
   We use the notation $\#(\cdot)$ to denote the cardinal number of a set, the sets under consideration will be sets with multiplicity,
  and then we  shall also count multiplicity when we use the notation $\#$.
  We use $\bullet$ or $\bullet\bullet$ to denote any possible positive integer. We shall use $\{a^{\thicksim k}\}$ to denote $\{\underbrace{a,\cdots,a}_{k\ times}\}$. For example,
  $
  \{a^{\thicksim 3},b^{\thicksim 2}\}=\{a,a,a,b,b\}.
  $
\end{ddef}
\begin{ddef}
  Let us use $\theta_1,\theta_2,\cdots,\theta_p$ to denote the spectrum of $F_1$ and denote the spectrum of $C([0,1],F_2)$ by $(t,i)$,
 where $0\leq t\leq1$ and $i\in \{1,2,\cdots,l\}$ indicates that it is in $i^{th}$ block of $F_2$. So
 $$
 Sp(C([0,1],F_2))=\coprod_{i=1}^{l}\{(t,i),\,0\leq t\leq1\}.
 $$
   Using identification of $f(0)=\varphi_0(a)$ and $f(1)=\varphi_1(a)$ for $(f,a)\in A,\,(0,i)\in Sp(C[0,1])$ is identified with
 $$
 (\theta_1^{\thicksim\alpha_{i1}},\theta_2^{\thicksim\alpha_{i2}},\cdots,\theta_p^{\thicksim\alpha_{ip}})\subset Sp(F_1)
 $$
 and $(1,i)\in Sp(C([0,1],F_2))$ is identified with
 $$
 (\theta_1^{\thicksim\beta_{i1}},\theta_2^{\thicksim\beta_{i2}},\cdots,\theta_p^{\thicksim\beta_{ip}})\subset Sp(F_1)
 $$
 as in $Sp(A)=Sp(F_1)\cup \coprod_{i=1}^{l}(0,1)_i$.
\end{ddef}
\begin{ddef}

 With $A=A(F_1,F_2,\varphi_0,\varphi_1)$ as above, let $\varphi:A\rightarrow M_n(\mathbb{C})$ be a homomorphism, then there exists a unitary $u$ such that
 $$
 \varphi(f,a)=u^*\cdot{\rm diag}\big(\underbrace{a(\theta_1),
 \cdots,a(\theta_1)}_{t_1},\cdots,\underbrace{a(\theta_p),\cdots,a(\theta_p)}_{t_p},
 f(y_1),\cdots,f(y_{\bullet}),0_{\bullet\bullet}\big)\cdot u,
 $$
 where $y_1,y_2,\cdots,y_{\bullet}\in \coprod_{i=1}^{l}[0,1]_i$. For $y=(0,i)$ (also denoted by $0_i$), one can replace $f(y)$ by
 $$\big(\underbrace{a(\theta_1),
 \cdots,a(\theta_1)}_{\alpha_{i1}},\cdots,\underbrace{a(\theta_p),\cdots,a(\theta_p)}_{\alpha_{ip}}\big)$$
 in the above expression, and do the same with $y=(1,i)$. After this procedure, we can assume each $y_k$ is strictly in the open interval $(0,1)_i$ for some $i$.
 We write the spectrum of $\varphi$ by
 $$
 Sp\varphi=\{\theta_1^{\thicksim t_1},\theta_2^{\thicksim t_2},\cdots,\theta_p^{\thicksim t_p},y_1,y_2,\cdots,y_{\bullet}\},
 $$
 where $y_k\in \coprod_{i=1}^{l}(0,1)_i$.

 If $f=f^*\in A$, we use $Eig(\varphi(f))$ to denote the eigenvalue list of $\varphi(f)$,  and then
 $$\#(Eig(\varphi(f)))=n\,\,{\rm(counting\,\,multiplicity)}.$$
\end{ddef}
\begin{ddef}\label{mappss}
  Let $A=A(F_1,F_2,\varphi_0,\varphi_1)\in \mathcal{C}$ be minimal. Written $a\in F_1$ as $a=(a(\theta_1),a(\theta_2)$, $\cdots,$ $a(\theta_p))$,
  $f(t)\in C([0,1],F_2)$ as $$f(t)=(f(t,1),f(t,2),\cdots,f(t,l))$$ where $a(\theta_j)\in M_{k_j}(\mathbb{C})$,
  $f(t,i)\in C([0,1],M_{l_i}(\mathbb{C}))$.

  For any $(f,a)\in A$ and $i\in\{1,2,\cdots,l\}$,
  define $\pi_t:A\rightarrow C([0,1],F_2)$ by $\pi_t(f,a)=f(t)$ and $\pi_t^i:A\rightarrow C([0,1],M_{l_i}(\mathbb{C}))$ by
  $\pi_t^i(f,a)=f(t,i)$ where $t\in(0,1)$ and $\pi_0^i(f,a)=f(0,i)$ (denoted by $\varphi_0^i(a)$), $\pi_1^i(f,a)=f(1,i)$
  (denoted by $\varphi_1^i(a)$). There is a canonical map $\pi_e:\,A\rightarrow F_1$ defined by $\pi_e((f,a))=a$, for all $j=\{1,2,\cdots,p\}$.
\end{ddef}
\begin{ddef}
  We use the convention that $A=A(F_1,F_2,\varphi_0,\varphi_1),\,B=B(F_1',F_2',\varphi_0',\varphi_1')$,
  where
  $$
  F_1=\bigoplus_{j=1}^p M_{k_j}(\mathbb{C}),\,\,F_2=\bigoplus_{i=1}^lM_{l_i}(\mathbb{C}),\,\,F_1'=\bigoplus_{j'=1}^{p'}M_{k_{j'}'}(\mathbb{C}),\,\,F_2'=\bigoplus_{i'=1}^{l'} M_{l_{i'}'}(\mathbb{C}).
  $$
  Set $L(A)=\sum_{i=1}^ll_i$, $L(B)=\sum_{i'=1}^{l'}l_{i'}'$. Denote $\{e_{ss'}^i\}(1\leq i\leq l,\,1\leq s,s'\leq l_i)$  the set of matrix units for $\bigoplus_{i=1}^l M_{l_i}({\mathbb{C}})$ and $\{f_{ss'}^j\}(1\leq j\leq p,\,1\leq s,s'\leq k_j)$ the set of matrix units for $\bigoplus_{j=1}^p  M_{k_j}({\mathbb{C}})$.
\end{ddef}
\begin{ddef}
  For each $\eta=\frac{1}{m}$ where $m\in\mathbb{N}_+$. Let $0=x_0<x_1<\cdots<x_{m}=1$ be a partition of $[0,1]$ into $m$ subintervals with equal length $\frac{1}{m}$. We will define a finite subset $H(\eta)\subset A_+$, consisting of two kinds of
 elements as described below.

 (a) For each subset $X_j=\{\theta_j\}\subset Sp(F_1)=\{\theta_1,\theta_2,\cdots,\theta_p\}$ and a list of integers $a_1,b_2,\cdots,a_l,b_l$ with
 $0\leq a_i<a_i+2\leq b_i\leq m$, denote $W_j\triangleq\coprod_{\{i|\alpha_{ij}\neq0\}}[0,a_i\eta]_i\cup\coprod_{\{i|\beta_{ij}\neq 0\}}[b_i\eta,1]_i$.
 Then we call $W_j$  the closed neighborhood of $X_j$, we define element $(f,a)\in A_+$ corresponding to $X_j\cup W_j$  as follows:

 Let $a=(a(\theta_1),a(\theta_2),\cdots,a(\theta_p))\in F_1$, where $a(\theta_j)=I_{k_j}$ and $a(\theta_s)=0_{k_s}$, if $s\neq j$. For each $t\in[0,1]_i,\,i=\{1,2,\cdots,l\}$, define
 $$
 f(t,i)=
 \begin{cases}
  \varphi_0^i(a)\dfrac{\eta-dist(t,[0,a_i\eta]_i)}{\eta}, & \mbox{if } 0\leq t\leq (a_i+1)\eta \\
  0, & \mbox{if } (a_i+1)\eta\leq t\leq (b_i-1)\eta \\
  \varphi_1^i(a)\dfrac{\eta-dist(t,[b_i\eta,1]_i)}{\eta}, & \mbox{if } (b_i-1)\eta\leq t\leq 1
\end{cases}
$$
All such elements $(f,a)=(f(t,1),f(t,2),\cdots,f(t,l))\in A_+$ are included in the set $H(\eta)$ and are called test functions of type 1.

(b) For each closed subset $X=\bigcup_s[x_{r_s},x_{r_{s+1}}]_i \subset [\eta,1-\eta]_i$ (the finite union of
closed intervals $[x_r,x_{r+1}]$ and points). So there are finite subsets for each $i$. Define $(f,a)$ corresponding to $X$ by $a=0$ and for each
$t\in(0,1)_r,r\neq i,f(t,r)=0$ and for $t\in(0,1)_i$ define
$$
f(t,i)=
\begin{cases}
  1-\dfrac{dist(t,X)}{\eta} , & \mbox{if } dist(t,X)<\eta \\
  0, & \mbox{if } dist(t,X)\geq \eta.
\end{cases}
$$
All such elements are called test functions of type 2.

Note that for any closed subset $Y\subset[\eta,1-\eta]$, there is a closed subset $X$ consisting of the union of the
intervals and points such that $X\supset Y$ and for any $x\in X$, $dist(x,Y)\leq\eta$.
\end{ddef}
\begin{ddef}\label{defg}
 Take $\eta$ as above, define a finite set $\widetilde{H}(\eta)$ as follows:

 In the construction of test functions of type 1, we may use $f_{ss'}^j\in F_1$ in place of $a\in F_1$, assume that all these elements
 are in $\widetilde{H}(\eta)$, and for all test functions $h\in H(\eta)$ of type 2, assume that all these elements $e_{ss'}^i\cdot h$ are in $\widetilde{H}(\eta)$.

 Then there exists a nature surjective map $\kappa: \widetilde{H}(\eta)\rightarrow H(\eta)$, for any subset $G\subset H(\eta)$, define a finite subset
 $\widetilde{G}\subset\widetilde{H}(\eta)$ by
 $$
 \widetilde{G}=\{\,h\,|\,h\in \widetilde{H}(\eta),\,\,\kappa(h)\in G\,\}.
 $$
\end{ddef}
\begin{ddef}
  Suppose $A$ is a $\mathrm{C}^*$-algebra, $B\subset A$ is a subalgebra, $F\subset A$ is a finite subset and let $\varepsilon>0$. If
 for each $f\in F$, there exists an element $g\in B$ such that $\|f-g\|<\varepsilon$, then we shall say that $F$ is approximately contained in $B$ to within $\varepsilon$, and denote this by $F\subset_{\varepsilon} B$.
\end{ddef}
  The following is clear by the standard techniques of spectral theory \cite{BBEK}.
\begin{lem} \label{rrpr}
  Let $A=\underrightarrow{lim}(A_n,\phi_{m,n})$ be an inductive limit of $\mathrm{C}^*$-algebras $A_n$ with morphisms
  $\phi_{m,n}:A_m\rightarrow A_n$. Then $A$ has $RR(A)=0$ if and only if for any finite self-adjoint subset $F\subset A_m$ and $\varepsilon >0$, there exists $n\geq m$ such
  that $$\phi_{m,n}(F)\subset_{\varepsilon}\{ f\in (A_n)_{sa}\,|\,f\,has\,finite\,spectrum\}.$$
\end{lem}
 The following is Lemma 2.3 in \cite{Liu}.
\begin{lem} \label{pair1}
  Let $A\in \mathcal{C}$, for any $1>\varepsilon >0$ and $\eta=\frac{1}{m}$ where $m\in\mathbb{N}_+$, if $\phi,\psi:A\rightarrow M_n(\mathbb{C})$ are unital homomorphisms with the condition that $Eig(\phi(h))$ and $Eig(\psi(h))$ can be paired to within $\varepsilon$ one by one for all $h\in H(\eta)$, then for each $i\in \{1,2,\cdots,l\}$, then there exists $X_i\subset Sp\phi\cap(0,1)_i$, $X_i'\subset Sp\psi\cap(0,1)_i$
   with $X_i\supset Sp\phi\cap[\eta,1-\eta]_i$, $X_i'\supset Sp\psi\cap[\eta,1-\eta]_i$ such that $X_i$ and $X_i'$ can be paired to within $2\eta$ one by one.
\end{lem}
\section{Main results}

 In this section, we will prove the following theorem.
 \begin{thrm} \label{mainn}
  Let $A=\underrightarrow{lim}(A_n,\phi_{m,n})$ be an inductive limit of Elliott-Thomsen algebras. Then one can write $A=\underrightarrow{lim}(B_n,\psi_{m,n})$, where all the $B_n$ are Elliott-Thomsen algebras, and all the homomorphisms $\psi_{m,n}$ are injective.
\end{thrm}
\begin{lem}[\cite{Li1}]\label{surjmap}
  Let $Y\subset [0,1]$ be a closed subset containing uncountably many points. Then
  there exists a surjective non-decreasing continuous map
  $$
  \rho: Y\rightarrow [0,1].
  $$
\end{lem}
\begin{ddef}\label{deff}
  Let $A=A(F_1,F_2,\varphi_0,\varphi_1)\in \mathcal{C}$ be minimal, the topology base on
  $$
  Sp(A)=\{\theta_1,\theta_2,\cdots,\theta_p\}\cup \coprod_{i=1}^{l}(0,1)_i
  $$
  at each point $\theta_j$ is given by
  $$
  \{\theta_j\}\cup\coprod_{\{i|\alpha_{ij}\neq0\}}(0,\varepsilon)_i\cup\coprod_{\{i|\beta_{ij}\neq0\}}(1-\varepsilon,1)_i.
  $$
  In general, this is a non Hausdorff topology.

  For closed subset $Y\subset Sp(A)$ and $\delta>0$, we will construct a space $Z$ and a continuous surjective map $\rho:Y\rightarrow Z$ such that $Z\cap (0,1)_i$ is a union of finitely many intervals for each $i\in \{1,2,\cdots,l\}$, and
  $dist(\rho(y),y)<\delta$ for all $y\in Y$. We can find a similar discussion in an old version of \cite{GL}.

  For any closed subset $Y\subset Sp(A)$, define index sets
  $$
  J_Y=\{j\,|\,\theta_j\in Y\},
  $$
  $$
  L_{0,Y}=\{i\,|\,(0,1)_i\cap Y=\varnothing\},
  $$
  $$
  L_{1,Y}=\{i\,|\,(0,1)_i\subset Y\},
  $$
  $$
  L_{l,Y}=\{i\,|\,i\notin L_{1,Y}\,\,and\,\,\exists\,s >0\,\,such\,\,that\,\,(0,s]_i\subset Y\},
  $$
  $$
  L_{ll,Y}=\{i\,|\,i\notin L_{1,Y}\cup L_{l,Y}\,\,and\,\,\exists\,\{y_n\}_{n=1}^{\infty}\subset (0,1)_i\cap Y\,\,such\,\,that\,\,\lim_{n\rightarrow\infty}y_n=0_i\},
  $$
  $$
  L_{r,Y}=\{i\,|\,i\notin L_{1,Y}\,\,and\,\,\exists \,t >0\,\,such\,\,that\,\,[1-t,1)_i\subset Y\},
  $$
  $$
  L_{rr,Y}=\{i\,|\,i\notin L_{1,Y}\cup L_{r,Y}\,\,and\,\,\exists\, \{y_n\}_{n=1}^{\infty}\subset (0,1)_i\cap Y\,\,such\,\,that\,\,\lim_{n\rightarrow\infty}y_n=1_i\},
  $$
  $$
  L_{a,Y}=\{i\,|\,i\notin L_{0,Y}\cup L_{1,Y}\}.
  $$
  Then we have
  $$
  L_{l,Y}\cup L_{ll,Y}\cup L_{r,Y}\cup L_{rr,Y}\subset L_{a,Y},
  $$
  $$
  L_{0,Y}\cup L_{1,Y} \cup L_{a,Y}=\{1,2,\cdots,l\}.
  $$

  Consider $Y\subset Sp(A)$, if $i\in L_{1,Y}\cup L_{l,Y}\cup L_{ll,Y}$, assume that $(0,i)\in Y$ and if $i\in L_{1,Y}\cup L_{r,Y}\cup L_{rr,Y}$, assume that $(1,i)\in Y$. For $\delta>0$, there exists $m\in \mathbb{N}_+$ such that $\frac{1}{m}<\frac{\delta}{2}$. Denote $Y_i=Y\cap [0,1]_i$, $i\in \{1,2,\cdots,l\}$, then we can construct a collection of finitely many points $\hat{Y}_i=\{y_1,y_2,\cdots\}\subset Y_i$ as below.

  (a). If $i\in L_{0,Y}$, let $\hat{Y}_i=\varnothing$;

  (b). If $i\in L_{1,Y}$, let $\hat{Y}_i=\{(0,i),(\frac{1}{m},i),\cdots,(1,i)\}$;

  (c). For each $i\in L_{a,Y}$, consider the set $Y_i\cap[\frac{r-1}{m},\frac{r}{m}]_i$, if $Y_i\cap[\frac{r-1}{m},\frac{r}{m}]_i\neq \varnothing$, then set
  $$
  x_i^r=\min\{x\,|\,x\in Y_i\cap[\frac{r-1}{m},\frac{r}{m}]_i\},
  $$
  $$
  \widetilde{x}_i^r=\max\{x\,|\,x\in Y_i\cap[\frac{r-1}{m},\frac{r}{m}]_i\}.
  $$
  Assume that $Y_i\cap[\frac{r-1}{m},\frac{r}{m}]_i\neq \varnothing$ iff $r\in \{r_1,r_2,\cdots,r_{\bullet}\}\subset \{1,2,\cdots,m\}$, then we have a finite set
  $$
  \{x_i^{r_1},\widetilde{x}_i^{r_1},x_i^{r_2},\cdots,x_i^{r_{\bullet}},
  \widetilde{x}_i^{r_{\bullet}}\}.
  $$
  Some of the points may be the same, we can delete the extra repeating points, and denote it by $\hat{Y}_i$.

   Denote $\hat{Y}=\coprod_{i=1}^l\hat{Y}_i$. Two points $(y_s,i),(y_t,i')\in\hat{Y}$ are said to be \textbf{adjacent}, if $(y_s,i),(y_t,i')$ are in the same interval (the case $i=i'$), and inside the open interval $(y_s,y_t)_i$, there is no other point in $\hat{Y}$. Note that if $\{(y_s,i),(y_t,i)\}$ is an adjacent pair and $(y_s,y_t)_i\cap Y\neq\varnothing$, then $dist((y_s,i),(y_t,i))<\delta$, and for any $y\in Y\cap\coprod_{i=1}^l[0,1]_i$, there exists $y'\in \hat{Y}$ such that $dist(y,y')<\delta$.

   It is obvious that $Y_i$ can be written as the union of $[y_s,y_t]_i\cap Y_i$, where $\{(y_s,i),$ $(y_t,i)\}$ runs over all adjacent pairs. We will define a space $Z$ and a
  continuous surjective map $\rho:Y\rightarrow Z$ as follows (see also \cite{Li1}).

  First, $Y\cap Sp(F_1)\subset Z$ and $Z$ contains a collection of finitely many  points $P(Z)=\{z_1,z_2,\cdots\}$, each $(z_s,i)\in P(Z)$ corresponding to one and only one
  $(y_s,i)\in \hat{Y}$. To define the edges of $Z$, we consider an
  adjacent pair $\{(y_s,i),(y_t,i)\}$. There are the following two cases.

  Case 1: If $[y_s,y_t]_i\cap Y$ has uncountably many points, then we let $Z$
    contain $[z_s,z_t]_i$, the line segment connecting $(z_s,i),(z_t,i)$. By Lemma \ref{surjmap}, there exists a non-decreasing surjective map
    $\rho:[y_s,y_t]_i\cap Y\rightarrow [z_s,z_t]_i$ such that $\rho((y_s,i))=(z_s,i), \rho((y_t,i))=(z_t,i)$. (Here both $[y_s,y_t]_i$ and $[z_s,z_t]_i$ are identified with interval $[0,1]$.)

  Case 2: If $[y_s,y_t]_i\cap Y$ has at most countably many points, then it is defined that there is no edge connecting $(z_s,i)$ and $(z_t,i)$. Since $[y_s,y_t]_i\cap Y$ is a
  countable closed subset of $[y_s,y_t]_i$, there exists an open interval
  $(y_s',y_t')_i\subset (y_s,y_t)_i$ such that $(y_s',y_t')_i\cap Y=\varnothing$.
  Let $\rho:[y_s,y_t]_i\cap Y\rightarrow \{(z_s,i),(z_t,i)\}$ be defined by
  $$
  \rho(y)=
  \begin{cases}
    (z_s,i), & \mbox{if } y\in [y_s,y_s']_i\cap Y  \\
    (z_t,i), & \mbox{if } y\in [y_t',y_t]_i\cap Y
  \end{cases}.
  $$

  By the above procedure for all adjacent pairs, we obtain a space $Z$ which satisfys that $Z\cap (0,1)_i$ is a union of finitely many intervals for each $i\in \{1,2,\cdots,l\}$.

  Notice that $\rho$ is defined on each $[y_s,y_t]_i\cap Y$ piece by piece, and $\rho((y_s,i))=(z_s,i)$ for each $s,i$, the definitions of $\rho$ on different pieces are consistent. Then we obtain a surjective map $\rho: Y\cap(0,1)_i \rightarrow Z\cap(0,1)_i$. Let $\rho: Y\cap Sp(F_1)\rightarrow Z\cap Sp(F_1)$ be defined by
  $\rho(\theta_j)=\theta_j$ for all $j\in J$.

  Then we obtain a surjective map $\rho:Y\rightarrow Z$, and we have $dist(\rho(y),y)<\delta$ for all $y\in Y$.
\end{ddef}
\begin{ddef}\label{adjj}
For any closed subset $X\subset Sp(A)$, denote that $A|_X=\{f|_{X}\,|\,f\in A\}$. For the ideal $I\subset A$, there exists a closed subset
  $Y\subset Sp(A)$ such that $I=\{f\in A\,|\,f|_{Y}=0\}$. Then $A/I\cong A|_{Y}$.
\end{ddef}
\begin{lem}\label{jihu}
  Let $A\in \mathcal{C}$ be minimal, $\varepsilon>0$, $Y\subset Sp(A)$ be a closed subset, $G\subset A|_Y$ be a finite subset. Suppose that $\delta>0$
  satisfys that, $dist(y,y')<\delta$ implies that $\|g(y)-g(y')\|<\varepsilon$ for all $g\in G$. Then there exists a closed subset $Z\subset Sp(A)$ and a surjective map $\rho:Y\rightarrow Z$ such that $A|_Z\in \mathcal{C}$ and $G\subset_{\varepsilon}A|_Z$, where $A|_Z$ is considered as a subalgebra of $A|_Y$ by the inclusion $\rho^*:A|_Z\rightarrow A|_Y$.
\end{lem}
\begin{proof}
  For closed subset $Y\subset Sp(A)$ and $\delta>0$, we can construct $Z$ and $\rho$ as in \ref{deff}.
  The surjective map $\rho:Y\rightarrow Z$ induces a homomorphism
  $$
  \rho^*:A|_Z\rightarrow A|_Y,
  $$
  $$
  (\rho^*(g))(y)=g(\rho(y)),\quad \forall y\in Y.
  $$
  Then we have
  $$
  \|\rho^*(g)-g\|=\max_{y\in Y}\|g(y)-g(\rho(y))\|<\varepsilon
  $$
  for any $g\in G$, then $G\subset_{\varepsilon}A|_Z$.

  We need to verify $A|_Z\in \mathcal{C}$. Define index sets for $Z$, we will have
  $$
  J_Z=J_Y,\,\,\,L_{0,Z}=L_{0,Y},
  $$
  $$
  L_{1,Z}\supset L_{1,Y},\,\,\,
  L_{ll,Z}=L_{rr,Z}=\varnothing.
  $$
  We will define positive numbers $s_i$ for all $i\in L_{l,Z}$, positive numbers $t_i$ for all $i\in L_{r,Z}$, and positive numbers $a_i<b_i$ for all $i\in L_{a,Z}$ to satisfy that $s_i<a_i<b_i$ (if $i\in L_{l,Z}$) and $a_i<b_i<t_i$ (if $i\in L_{r,Z}$) as below.

  For $i\in L_{l,Z}$, let $s_i=\max\{s\,|\,(0,s]_i\subset Z\}$. For $i\in L_{r,Z}$, let $t_i=\min\{t\,|\,[t,1)_i\subset Z\}$. Note that if $i\in L_{l,Z}\cap L_{r,Z}$, then $s_i<t_i$.

%%%%  Denote$$L_{lr}=\{i\,|\,i\in L_l\,\,and\,\,\exists\,\{y_n\}_{n=1}^{\infty}\subset %%%%(s_i,1)_i\cap %%%%Y\,\,such\,\,that\,\,\lim_{n\rightarrow\infty}y_n=s_i\},$$$$L_{rl}=\{i\,|\,i\in %%%%L_r\,\,and\,\,\exists\, \{y_n\}_{n=1}^{\infty}\subset (0,t_i)_i\cap %%%%Y\,\,such\,\,that\,\,\lim_{n\rightarrow\infty}y_n=t_i\}.$$

  For $i\in L_{l,Z}$, choose $a_i$ with $s_i<a_i<1$ such that $(s_i,a_i)_i\cap Y=\varnothing$. For $i\in L_{a,Z}\backslash L_{l,Z}$, choose $a_i$ with $0<a_i<\delta$ such that $(0,a_i)_i\cap Y=\varnothing$ (we don't need to define $s_i$ at this case).
  Evidently the numbers $a_i$ satisfies that $a_i<t_i$ provided $i\in L_{r,Z}$.

  For $i\in L_{r,Z}$, choose $b_i$ with $a_i<b_i<t_i$ such that $(b_i,t_i)_i\cap Y=\varnothing$. For $i\in L_{a,Z}\backslash L_{r,Z}$, choose $b_i$ with $b_i>1-\delta$ such that $(b_i,1)_i\cap Y=\varnothing$ (we don't need to define $t_i$ in this case).

  Define closed subsets of $Sp(A)$ as below:
  $$
  Z_1=\coprod_{i\in L_{a,Z}}[a_i,b_i]_i,
  $$
  $$
  Z_2=\{\theta_j,\,j\in J\}\cup\coprod_{i\in L_{1,Z}}(0,1)_i\cup\coprod_{i\in L_{l,Z}}(0,s_i]_i\cup\coprod_{i\in L_{r,Z}}[t_i,1)_i,
  $$
 % $$ X_3 = X_2 \cup \bigcup_{i\in L_{ll}}(0,s_i]_i\cup\bigcup_{i\in L_{lr}} [s_i,s_i']_i\cup \bigcup_{i\in L_{rr}}[t_i,1)_i\cup\bigcup_{i\in L_{rl}} (t_i',t_i]_i.$$
  Then $Z_1\cap Z_2=\varnothing$ and $Z\subset Z_1\cup Z_2$, we have
  $A|_Z\cong A|_{Z_2}\oplus A|_{Z_1}$, where $A|_{Z_1}$ is a direct sum of matrices over interval algebras or matrix algebras.

  Now we consider $A|_{Z_2}$, for each $i\in L_{l,Z}$, we denote $F_2^i=M_{l_i}(\mathbb{C})$ by $F_{2,l}^i$; and for each $i\in L_{r,Z}$, we denote
  $F_2^i=M_{l_i}(\mathbb{C})$ by $F_{2,r}^i$.
  Let
  $$
  E_1=\bigoplus_{j\in J_Z}F_1^j\oplus\bigoplus_{i\in L_{l,Z}}F_{2,l}^i\oplus\bigoplus_{i\in L_{r,Z}}F_{2,r}^i
  $$
  $$
  E_2=\bigoplus_{i\in L_{1,Z}}F_2^i\oplus\bigoplus_{i\in L_{l,Z}}F_{2,l}^i\oplus\bigoplus_{i\in L_{r,Z}}F_{2,r}^i.
  $$

  Written $a\in F_1$ by $a=(a(\theta_1),a(\theta_2),\cdots,a(\theta_p))$. Define $\pi:F_1\rightarrow F_1$ by
  $$
  \pi(a)=a'=(a'(\theta_1),a'(\theta_2),\cdots,a'(\theta_p)),
  $$
  where
  $$
  a'(\theta_j)=
  \begin{cases}
    a(\theta_j), & \mbox{if } j\in J_Z \\
    0_{k_j}, & \mbox{if } j\notin J_Z.
  \end{cases}
  $$
  Then there exist a natural inclusion $\iota$ and a projection $\iota^*$
  such that
  $$
  \iota\circ\iota^*=\pi:F_1\rightarrow F_1,
  $$
  $$
  \iota^*\circ\iota=id: \bigoplus_{j\in J_Z}F_1^j\rightarrow \bigoplus_{j\in J_Z}F_1^j.
  $$
  Then we have if $i\in L_{1,Z}\cup L_{l,Z}$, then
  $\varphi_0^i(a)=\varphi_0^i(\pi(a))$ for any $a\in F_1$,
  and if $i\in L_{1,Z}\cup L_{r,Z}$, then
  $\varphi_1^i(a)=\varphi_1^i(\pi(a))$ for any $a\in F_1$.

  Let $\psi_0:E_1\rightarrow E_2$ be defined as follows:

  (1). For the part $\mathop{\bigoplus}\limits_{j\in J_Z}F_1^j$ in $E_1$, the partial map of $\psi_0$
  is defined to be
  $$
  \bigoplus_{i\in L_{1,Z}}\varphi_0^i\circ\iota\oplus\bigoplus_{i\in L_{l,Z}}\varphi_0^i\circ\iota\oplus\bigoplus_{i\in L_{r,Z}}0
  $$

  (2). For the part $\mathop{\bigoplus}\limits_{i\in L_{l,Z}}F_{2,l}^i$ in $E_1$, the partial map of $\psi_0$ is zero;

  (3). For the part $\mathop{\bigoplus}\limits_{i\in L_{r,Z}}F_{2,r}^i$ in $E_1$, the partial map of $\psi_0$ is defined to be
  $$
  \bigoplus_{i\in L_{1,Z}}0\oplus\bigoplus_{i\in L_{l,Z}}0\oplus\bigoplus_{i\in L_{r,Z}}id_i
  $$
  where $id_i$ ($i\in L_{r,Z}$) is the identity map from $M_{l_i}(\mathbb{C})$ to $M_{l_i}(\mathbb{C})$.

  Similarly, let $\psi_1:E_1\rightarrow E_2$ be defined as follows:

   (1). For the part $\mathop{\bigoplus}\limits_{j\in J_Z}F_1^j$ in $E_1$, the partial map of $\psi_1$ is defined to be
  $$
  \bigoplus_{i\in L_{1,Z}}\varphi_1^i\circ\iota\oplus\bigoplus_{i\in L_{l,Z}}0\oplus\bigoplus_{i\in L_{r,Z}}\varphi_1^i\circ\iota;
  $$

  (2). For the part $\mathop{\bigoplus}\limits_{i\in L_{l,Z}}F_{2,l}^i$ in $E_1$, the partial map of $\psi_0$ is defined to be
  $$
  \bigoplus_{i\in L_{1,Z}}0\oplus\bigoplus_{i\in L_{l,Z}}id_i\oplus\bigoplus_{i\in L_{r,Z}}0;
  $$
  where $id_i$ ($i\in L_{l,Z}$) is the identity map from $M_{l_i}(\mathbb{C})$ to $M_{l_i}(\mathbb{C})$.

  (3). For the part $\mathop{\bigoplus}\limits_{i\in L_{r,Z}}F_{2,r}^i$ in $E_1$, the partial map of $\psi_0$ is zero.

  Evidently $A|_{Z_2}\cong B(E_1,E_2,\psi_0,\psi_1)\in \mathcal{C}$, then we have $A|_Z\in \mathcal{C}$.
\end{proof}
  Using some similar techniques in \cite{Su}, we will have some perturbation results.
\begin{lem}\label{ptb1}
  Let $A=A(F_1,F_2,\varphi_0,\varphi_1)\in \mathcal{C}$ be minimal, $B=M_n(\mathbb{C})$, $F\subset A$ be a finite subset. Given $1>\varepsilon>0$, there exist $\eta,\varepsilon'>0$
   such that,  if $\phi,\psi:A\rightarrow B$ are unital homomorphisms satisfy the following conditions:

  (1) $Sp\phi=Sp\psi$;

  (2) $\|\phi(h)-\psi(h)\|<\varepsilon'$, $\forall$ $h\in H(\eta)\cup \widetilde{H}(\eta)$,

  then there is a continuous path of homomorphisms $\phi_t:A\rightarrow B$ such that $\phi_0=\phi$, $\phi_1=\psi$ and
  $$
  \|\phi_t(f)-\phi(f)\|<\varepsilon
  $$
  for all $f\in F$, $t\in [0,1]$.
  %Moreover, $\phi_t$ has generalized standard form.
\end{lem}
\begin{proof}
  Without loss of generality, we may suppose that for each $f\in F$, $\|f\|\leq 1$. Since $F\subset A$ is a finite set, there exists an integer $m>0$ such that for any $dist(x,x')<\frac{2}{m}$, $\|f(x)-f(x')\|<\frac{\varepsilon}{2}$
  holds for all $f\in F$, and $\varepsilon'$ will be specified later.  Set $\eta=\frac{1}{2mn}$, then we have finite subsets
  $H(\eta)$ and $\widetilde{H}(\eta)$.

  There exist unitaries $U,V$ such that
  $$
  \phi(f,a)=U^*\phi'(f,a) U,
  \quad
  \psi(f,a)=V^*\phi'(f,a) V.
  $$
  here we denote $\phi':A\rightarrow B$ by
  $$
  \phi'(f,a)={\rm diag}\big( a(\theta_1)^{\sim t_1},\cdots,a(\theta_p)^{\sim t_p},f(x_1),f(x_2),\cdots,f(x_{\bullet})\big)
  $$
  where $x_1,x_2,\cdots\in \coprod_{i=1}^{l}(0,1)_i$.

  Divide $(0,1)_i$ into $2mn$ intervals of equal length $\frac{1}{2mn}$, for each sub-interval $(\frac{k-1}{m},\frac{k}{m})_i$, $k=1,2,\cdots,m$, there exist an integer $a_k^i$ such that
  $$
  (a_k^i\eta,a_k^i\eta+2\eta)_i\subset (\frac{k-1}{m},\frac{k}{m})_i\,\,\, {\rm and}\,\,\,
  (a_k^i\eta,a_k^i\eta+2\eta)_i\cap Sp\phi=\varnothing.
  $$
  Then we have
  $$
  Sp\phi'=Sp\phi'\cap \coprod_{i=1}^l \big([0,a_1^i\eta]_i\cup[a_m^i\eta+2\eta,1]_i\cup\bigcup_{k=1}^{m-1}[a_k^i\eta+2\eta,a_{k+1}^i\eta]_i\big).
  $$

  For each $X_j=\{\theta_j\}$ and
  $W_j\triangleq\coprod_{\{i|\alpha_{ij}\neq0\}}[0,a_1^i\eta]_i\cup\coprod_{\{i|\beta_{ij}\neq 0\}}[a_m^i\eta+2\eta,1]_i$,
  we can define $h_j$ corresponding to $X_j\cup W_j$ for all $j\in\{1,2,\cdots,p\}$, and we can define $h_k^i$ corresponding to
  $[a_k^i\eta+2\eta,a_{k+1}^i\eta]_i$ for each $i\in \{1,2,\cdots,l\}$, $k\in\{1,2,\cdots,m-1\}$.

  Denote
  $$
  G=\{h_1,h_2,\cdots,h_p,h_1^1,\cdots,h_{m-1}^1,\cdots,h_1^l,\cdots,h_{m-1}^l\},
  $$
  We will construct $\widetilde{G}$ as in \ref{defg}:
  $$
  \widetilde{G}=\{\,h\,|\,h\in \widetilde{H}(\eta),\,\,\kappa(h)\in G\,\}.
  $$

  To define $\phi'':A\rightarrow B$,
  change all the elements $x\in Sp\phi'\cap(0,a_1^i\eta]_i$ to $0_i\thicksim \{\theta_{1}^{\thicksim\alpha_{i1}},\cdots,
  \theta_p^{\thicksim\alpha_{ip}}\}$ and $x \in Sp\phi'\cap(a_m^i\eta+2\eta,1)_i$ to
  $1_i\thicksim \{\theta_1^{\thicksim\beta_{i1}},\cdots,\theta_p^{\thicksim\beta_{ip}}\}$, change all the elements
  $x\in Sp\phi'\cap[a_{k-1}^i\eta+2\eta,a_k^i\eta]_i$ to $(\frac{k-1}{m},i)\in [a_{k-1}^i\eta+2\eta,a_k^i\eta]_i$ for each $i\in \{1,2,\cdots,l\}$, $k\in\{2,\cdots,m\}$. Set $\omega_k^i=\#(Sp\phi'\cap[a_{k-1}^i\eta+2\eta,a_k^i\eta]_i)$.

  There exists a unitary $W$ such that
  $$
   W\phi''(f)W^*=
  \left(
  \begin{array}{ccc}
    a(\theta_1)\otimes I_{t_1'(x)} &  &       \\
    \\
    \quad\quad\quad\quad\ddots &  &        \\
    \\
    & a(\theta_p)\otimes I_{t_p'(x)}&         \\
    \\
     & \quad\quad\quad f((\frac{1}{m},1))\otimes I_{\omega_1^1} &     \\
   \\
    & \quad\quad\quad\quad\quad\quad\ddots &        \\
    \\
    &  & f((\frac{m-1}{m},l))\otimes I_{\omega_m^l}
  \end{array}
  \right).
  $$

  From the construction of $\phi''$, we have
  $$
  \phi'(h)=\phi''(h),\quad\forall\,h\in G\cup \widetilde{G}.
  $$
  Let $P_j= W\phi'(h_j)W^*$, $P_k^i=W\phi'(h_k^i)W^*$, then $P_1,\cdots,P_p,P_1^1,\cdots,P_1^l,\cdots$,$P_{m-1}^l$ are projections, some of them
  may be zero, we  delete them and rewrite them by $P_1,\cdots,P_{n'}$, note that $n'\leq n$ and we can write
  $$
  P_1=
  \left(
  \begin{array}{cccc}
    I_{r_1} &  &  &  \\
     & 0 &  &  \\
     &  & \ddots &  \\
     &  &  & 0
  \end{array}
  \right)
  ,\,\cdots\,,
  P_{n'}=
  \left(
  \begin{array}{cccc}
    0 &  &  &  \\
     & 0 &  &  \\
     &  & \ddots &  \\
     &  &  & I_{r_{n'}}
  \end{array}
  \right)
  .$$

  Since
  $$
  \|\phi(h)-\psi(h)\|<\varepsilon',\quad \forall \,h\in H(\eta)\cup \widetilde{H}(\eta),
  $$
  then we have the following inequality:
  $$
  \|U^*W^*P_rWU-V^*W^*P_rWV\|<\varepsilon',\quad r=1,2,\cdots,n'.
  $$
  Set $\widetilde{W}=WVU^*W^*$, let us write the unitary
  $\widetilde{W}=
  \left(
  \begin{array}{cc}
    w_{11} & w_{1*} \\
    w_{*1} & w_{**}
  \end{array}
  \right)
  $,
  where the size of $w_{11}$ is the same as the rank of $P_1$, then we have $\|w_{1*}\|<\varepsilon'$ and $\|w_{*1}\|<\varepsilon'$,
  apply this computation to $P_2,\cdots,P_{n'}$, then we have
  $$
  \|\widetilde{W}-
  \left(
  \begin{array}{ccc}
    w_{11} &  &  \\
     & \ddots &  \\
     &  & w_{n'n'}
  \end{array}
  \right)
  \|<{n'}^2\varepsilon'\leq n^2\varepsilon'
  $$

  Write $T=
  \left(
  \begin{array}{ccc}
    w_{11} &  &  \\
     & \ddots &  \\
     &  & w_{n'n'}
  \end{array}
  \right)$,
  $T$ is invertible if $n^2\varepsilon'<1$, there is a unitary $S$ such that $T=|T^*|S$, so
  $$
  \|\widetilde{W}S^*-|T^*|\|<n^2\varepsilon'
  $$
  Since $\widetilde{W} S^*$ is a unitary and $|T^*|$ is close to $I$ to within $n^2\varepsilon'$, we have
  $$
  \|\widetilde{W}S^*-I\|\leq\|\widetilde{W}S^*-|T^*|\|+\||T^*|-I\|< 2n^2\varepsilon'.
  $$
  Let $R_t$ ($t\in[\frac{2}{3},1]$) be a unitary path in a $2n^2\varepsilon'$ neighbourhood of $I$ such that $R_{\frac{2}{3}}=\widetilde{W}S^*$ and $R_1=I$.

  Since
  $$
  \|U^*W^*(W\phi'(h)W^*)WU-V^*W^*(W\phi'(h)W^*)WV\|<\varepsilon',\,\,\forall \,h\in H(\eta)\cup \widetilde{H}(\eta).
  $$
  Then we have
  $$
  \|U^*W^*(W\phi'(h)W^*)WU-V^*W^*R_t(W\phi'(h)W^*)R^*_tWV\|<4 n^2\varepsilon'+\varepsilon'<5n^2\varepsilon',
  $$
  for all $h\in H(\eta)\cup \widetilde{H}(\eta)$, $t\in [\frac{2}{3},1]$, when $t=\frac{2}{3}$, we have
  $$
  \|S(W\phi'(h)W^*)-(W\phi'(h)W^*)S\|<5n^2\varepsilon',\,\,\,\forall \,h\in H(\eta)\cup \widetilde{H}(\eta).
  $$
  For any $h\in G\cup\widetilde{G}$, we have $\phi'(h)=\phi''(h)$, then
  $$
  \|S(W\phi''(h)W^*)-(W\phi''(h)W^*)S\|<5n^2\varepsilon',\,\,\, \forall\, h\in G\cup \widetilde{G}.
  $$

  Recall that $S$ has diagonal form $S={\rm diag}(S_1,\cdots,S_{n'})$, write $S=(w_{st}^r)$ as
  $$
  S=
  \left(
  \begin{array}{ccc}
  \left(
  \begin{array}{ccc}
   w_{11}^1 & \cdots & w_{1r_1}^1  \\
   \vdots & \ddots &  \vdots\\
  w_{r_11}^1  & \cdots  & w_{r_1r_1}^1
  \end{array}
  \right)
   &  &  \\
     & \ddots &  \\
     &  &
   \left(
  \begin{array}{ccc}
   w_{11}^{n'} & \cdots  & w_{1r_{n'}}^{n'}\\
   \vdots & \ddots &  \vdots\\
  w_{r_n'1}^{n'} & \cdots  & w_{r_{n'}r_{n'}}^{n'}
  \end{array}
  \right)
  \end{array}
  \right).
  $$
  Then for  the matrix $w_{st}^r$, it commutes with the matrix units to within $5n^2\varepsilon'$, so there exist
$d_{st}^r\in \mathbb{C}$ such that
 $$
 \|w_{st}^r-d_{st}^rI_{st}^r\|<5n^4\varepsilon',
 $$
 where $I_{st}^r$ is the identity matrix with suitable size. Write $D=(d_{st}^rI_{st}^r)$ as
  $$
  \left(
  \begin{array}{ccc}
  \left(
  \begin{array}{ccc}
   d_{11}^1I_{11}^1 & \cdots & d_{1r_1}^1I_{1r_1}^1  \\
   \vdots & \ddots &  \vdots\\
  d_{r_11}^1I_{r_11}^1  & \cdots  & d_{r_1r_1}^1I_{r_1r_1}^1
  \end{array}
  \right)
   &  &  \\
     & \ddots &  \\
     &  &
   \left(
  \begin{array}{ccc}
   d_{11}^{n'}I_{11}^{n'} & \cdots  & d_{1r_n'}^{n'}I_{1r_n'}^{n'}\\
   \vdots & \ddots &  \vdots\\
  d_{r_n'1}^{n'} I_{r_n'1}^{n'} & \cdots  & d_{r_{n'}r_{n'}}^{n'}I_{r_{n'}r_{n'}}^{n'}
  \end{array}
  \right)
  \end{array}
  \right)
  $$
 Then we have
 $$
 \|S-D\|<5n^6\varepsilon',
 $$
 $$
 D(W\phi''(f)W^*)=(W\phi''(f)W^*)D,\quad \forall\,\,f\in A.
 $$
 Hence,
 $$
 \|D(W\phi'(f)W^*)-(W\phi'(f)W^*)D\|<2\|D\|\varepsilon'<2(1+5n^6\varepsilon')\varepsilon'< 12 n^6\varepsilon',\,\,\forall\,\,f\in F.
 $$

 Decompose $D=|D^*|O$ in the commutant of $W\phi''(f)W^*$. Let $R_t'\,(t\in [\frac{1}{3},\frac{2}{3}])$ be an exponential unitary path in that commutant such that $R_{\frac{1}{3}}'=O^*$ and $R_{\frac{2}{3}}'=I$.

 Notice that
 $$
 \|S^*O^*-|D^*|\|<5n^6\varepsilon',
 $$
 use the same technique above, we have
 $$
 \|S^*O^*-I\|<10n^6\varepsilon',
 $$
 Hence there is a unitary path $R_t^{''}\,(t\in[0,\frac{1}{3}])$ in a $10n^6\varepsilon'$ neighbourhood of $I$ such
 that $R_0^{''}= I$ and $R_{\frac{1}{3}}^{''}=S^*O^*$.

 Finally, choose $\varepsilon'$ such that $4n^2\varepsilon'+12n^6\varepsilon'+20n^6\varepsilon'<\varepsilon$, we may take
 $\varepsilon'$ to be $\frac{\varepsilon}{40n^6}$, define a unitary path $u_t$ on $[0,1]$ as follows:
 $$
 u_t^*=
 \begin{cases}
   U^*W^*R_t^{''}W, & \mbox{if } t\in [0,\frac{1}{3}] \\
   U^*W^*S^*R_t'W, & \mbox{if } t\in [\frac{1}{3},\frac{2}{3}]. \\
   V^*W^*R_tW, & \mbox{if } t\in [\frac{2}{3},1]
 \end{cases}
 $$

  Denote
  $$
  \phi_t(f)=u_t^*\cdot
  {\rm diag}\big( a(\theta_1)^{\sim t_1},\cdots,a(\theta_p)^{\sim t_p},f(x_1),f(x_2),\cdots,f(x_{\bullet})\big)
  \cdot u_t.
  $$
  Then $\phi_0=\phi$, $\phi_1=\psi$, $u_0=U$, $u_1=V$ and we will have
  $$
  \|\phi_t(f)-\phi(f)\|<\varepsilon
  $$
 for all $f\in F$, $t\in [0,1]$.
 \end{proof}
\begin{lem}\label{ptb2}
  Let $A=A(F_1,F_2,\varphi_0,\varphi_1)\in \mathcal{C}$ be minimal, $B=M_n(\mathbb{C})$, $F\subset A$ be a finite subset. Given $1>\varepsilon>0$, there exist
   $\eta,\eta_1,\varepsilon'>0$, such that if $\phi,\psi:A\rightarrow B$ are unital homomorphisms satisfy the following conditions:

  (1) $\|\phi(h)-\psi(h)\|<1$, $\forall$ $h\in H(\eta_1)$;

  (2) $\|\phi(h)-\psi(h)\|<\frac{\varepsilon'}{8}$, $\forall$ $h\in H(\eta)\cup \widetilde{H}(\eta)$,

  then there is a continuous path of homomorphisms $\phi_t:A\rightarrow B$ such that $\phi_0=\phi$, $\phi_1=\psi$ and
  $$
  \|\phi_t(f)-\phi(f)\|<\varepsilon
  $$
  for all $f\in F$, $t\in [0,1]$. Moreover, for each $y\in (Sp\phi\cup Sp\psi)\cap\coprod_{i=1}^{l}(0,1)_i$, we have
  $$
  \overline{B_{4\eta_1}(y)}\subset \bigcup_{t\in [0,1]}Sp\phi_t,
  $$
  where $\overline{B_{4\eta_1}(y)}=\{x\in \coprod_{i=1}^{l}[0,1]_i:\, dist(x,y)\leq 4\eta_1\}$.
 \end{lem}
\begin{proof}
  Take $\varepsilon',\eta,m$ as in  Lemma \ref{ptb1}. Let $\eta_1=\frac{1}{m_1}<\frac{\eta}{2}$ satisfys that $\|h(x)-h(x')\|<\frac{\varepsilon'}{8}$ for any
 $dist(x,x')\leq 4\eta_1$ and for all $h\in H(\eta)\cup \widetilde{H}(\eta)$.

  There exist unitaries $U,V$ such that
  $$
  \phi(f,a)=U^*\cdot{\rm diag}\big( a(\theta_1)^{\sim s_1},\cdots,a(\theta_p)^{\sim s_p},f(x_1),f(x_2),\cdots,f(x_{\bullet})\big)\cdot U.
  $$
  $$
  \psi(f,a)=V^*\cdot{\rm diag}\big( a(\theta_1)^{\sim t_1},\cdots,a(\theta_p)^{\sim t_p},f(y_1),f(y_2),\cdots,f(y_{\bullet\bullet})\big)\cdot V.
  $$
  where $f\in A$, $x_1,x_2,\cdots,y_1,y_2,\cdots\in \coprod_{i=1}^{l}(0,1)_i$.

  From condition (1) and Lemma \ref{pair1}, for each $i\in \{1,2,\cdots,l\}$, there exists $X_i\subset Sp\phi\cap(0,1)_i$, $X_i'\subset Sp\psi\cap(0,1)_i$
  with $X_i\supset Sp\phi\cap[\eta_1,1-\eta_1]_i$ , $X_i'\supset Sp\psi\cap[\eta_1,1-\eta_1]_i$ such that $X_i$ and $X_i'$ can be paired to
  within $2\eta_1$ one by one, denote the one to one correspondence by $\pi: X_i\rightarrow X_i'$.

  To define $\phi'$,
  change all the elements $x_k\in (0,\eta_1)_i\backslash X_i$ to $0_i\thicksim \{\theta_{1}^{\thicksim\alpha_{i1}},\cdots,
  \theta_p^{\thicksim\alpha_{ip}}\}$ and $x_k\in (1-\eta_1,1)_i\backslash X_i$ to
  $1_i\thicksim \{\theta_1^{\thicksim\beta_{i1}},\cdots,\theta_p^{\thicksim\beta_{ip}}\}$,
  and finally, change all the $x_k\in X_i$ to $\pi(x_k)\in X_i'$.
  To define $\psi'$, change all the elements $y_k\in (0,\eta_1)_i\backslash X_i'$ to
  $0_i\thicksim \{\theta_{1}^{\thicksim\alpha_{i1}},\cdots,\theta_p^{\thicksim\alpha_{ip}}\}$
  and $y_k\in (1-\eta_1,1)_i\backslash X_i'$ to $1_i\thicksim \{\theta_1^{\thicksim\beta_{i1}},\cdots,\theta_p^{\thicksim\beta_{ip}}\}$.
  Then we have
  $$
  Sp\phi'\cap(0,1)_i=Sp\psi'\cap(0,1)_i
  $$
  for all $i=1,2,\cdots,l$.

  Since $2\eta_1<\eta=\frac{1}{2mn}$, then for each $[0,1]_i$, there exist integers $a_i, b_i$ with $1<a_i<a_i+2\leq b_i< m_1$ such that
  $$
  Sp\phi\cap(a_i\eta_1,b_i\eta_1)_i=Sp\psi\cap(a_i\eta_1,b_i\eta_1)_i=\varnothing.
  $$

  Then for $X_j=\{\theta_j\}$ and $W_j\triangleq\coprod_{\{i|\alpha_{ij}\neq0\}}[0,a_i\eta_1]_i\cup\coprod_{\{i|\beta_{ij}\neq 0\}}[b_i\eta_1,1]_i$,
  we can define $h_j$ corresponding to $X_j$ and $W_j$ in $H(\eta_1)$, then $\phi(h_j),\psi(h_j)$ are projections and
  $$
  \phi(h_j)=\phi'(h_j),\quad \psi(h_j)=\psi'(h_j),
  $$
  $$
  \|\phi(h_j)-\psi(h_j)\|<1,
  $$
  for each $j=1,2,\cdots,p$, this fact means that
  $$
  Sp\phi'\cap Sp(F_1)=Sp\psi'\cap Sp(F_1).
  $$
  Now we have $ Sp\phi'=Sp\psi'$.

  For each $x_k\in Sp\phi\cap (0,1)_i$, define a continuous map
  $$
  \gamma_k:[0,\frac{1}{3}]\rightarrow \coprod_{i=1}^{l}\,[0,1]_i
  $$
  with the following properties:

  (i) $\gamma_k(0)=x_k$;

  (ii) $
  \gamma_k(\frac{1}{3})=
  \begin{cases}
    0_i, & \mbox{if } x_k\in (0,\eta_1)_i\backslash X_i \\
    \pi(x_k), & \mbox{if } x_k\in X_i \\
    1_i, & \mbox{if } x_k\in (1-\eta_1,1)_i\backslash X_i
  \end{cases};
  $

  (iii) Im$\gamma_k=\overline{B_{4\eta_1}(x_k)}=\{x\in \coprod_{i=1}^{l}[0,1]_i;\, dist(x,x_k)\leq 4\eta_1\}$.

  Define $\phi_t$ on $[0,\frac{1}{3}]$ by
   $$
   \phi_t(f)=U^*\cdot{\rm diag}\big(a(\theta_1)^{\sim s_1},\cdots,a(\theta_p)^{\sim
   s_p},f(\gamma_1(x)),f(\gamma_2(x)),\cdots,f(\gamma_{\bullet}(x))\big)\cdot U.
   $$
  Then $\phi_{\frac{1}{3}}=\phi'$, and
  $$
  \|\phi(h)-\phi'(h)\|< \frac{\varepsilon'}{8},\quad \forall\, h\,\in H(\eta)\cup \widetilde{H}(\eta).
  $$
  Similarly, for each $y_k\in Sp\psi\cap (0,1)_i$, define a continuous map
  $$
  \gamma_k':[\frac{2}{3},1]\rightarrow \coprod_{i=1}^{l}\,[0,1]_i
  $$
  with the following properties:

  (i) $
  \gamma_k'(\frac{2}{3})=
  \begin{cases}
    0_i, & \mbox{if } y_k\in (0,\eta_1)_i\backslash X_i' \\
    y_k, & \mbox{if } y_k\in X_i' \\
    1_i, & \mbox{if } y_k\in (1-\eta_1,1)_i\backslash X_i'
  \end{cases};
  $

  (ii) $\gamma_k'(1)=y_k$;

  (iii) Im$\gamma_k'=\overline{B_{4\eta_1}(y_k)}=\{y\in \coprod_{i=1}^{l}[0,1]_i;\, dist(y,y_k)\leq 4\eta_1\}$.

  Define $\phi_t$ on $[\frac{2}{3},1]$ by
   $$
   \phi_t(f)=V^*\cdot{\rm diag}\big(a(\theta_1)^{\sim t_1},\cdots,a(\theta_p)^{\sim
   t_p},f(\gamma_1'(y)),f(\gamma_2'(y)),\cdots,f(\gamma_{\bullet\bullet}'(y))\big)\cdot V.
   $$
  Then $\phi_{\frac{2}{3}}=\psi'$, and
  $$
  \|\psi(h)-\psi'(h)\|< \frac{\varepsilon'}{8},\quad \forall\, h\in H(\eta)\cup \widetilde{H}(\eta).
  $$
  $$
  \|\phi'(h)-\psi'(h)\|<\frac{\varepsilon'}{8}+\frac{\varepsilon'}{8}+\frac{\varepsilon'}{8}< \frac{\varepsilon'}{2},\quad \forall\, h\in H(\eta)\cup \widetilde{H}(\eta).
  $$
  Apply Lemma \ref{ptb1},  then there is a continuous path of homomorphisms $\phi_t:A\rightarrow B$, $t\in [\frac{1}{3},\frac{2}{3}]$, such that $\phi_{\frac{1}{3}}=\phi'$,
  $\phi_{\frac{2}{3}}=\psi'$ and
  $$
  \|\phi_t(f)-\phi'(f)\|<\frac{\varepsilon}{2},\quad \forall\,\,f\in F.
  $$
  Now we have a continuous path of homomorphisms $\phi_t:A\rightarrow B$ such that $\phi_0=\phi$, $\phi_1=\psi$ and
  $$
  \|\phi_t(f)-\phi(f)\|<\varepsilon
  $$
  for all $f\in F$, $t\in [0,1]$.

  From the property (iii) of $\gamma_k$ and  $\gamma_k'$, for any $y\in (Sp\phi\cup Sp\psi)\cap \coprod_{i=1}^{l}(0,1)_i$, we have
  $$
  \overline{B_{4\eta_1}(y)}\subset \bigcup_{t\in [0,1]}Sp\phi_t.
  $$
  where $\overline{B_{4\eta_1}(y)}=\{x\in \coprod_{i=1}^{l}[0,1]_i:\, dist(x,y)\leq 4\eta_1\}$.
\end{proof}
\begin{thrm}\label{mainthrm}
  Let $A,B\in \mathcal{C}$, $F\subset A$ be a finite subset, $Y\subset Sp(B)$ be a closed subset, and $G\subset B|_Y$ be a finite subset.
  Let $\phi:A\rightarrow B|_Y$ be a unital injective homomorphism, then for any $\varepsilon>0$, there exist a closed subset $Z\subset Y$
  and a unital injective homomorphism $\psi:A\rightarrow B|_Z$ such that,

  (1) $\|\phi(f)-\psi(f)\|<\varepsilon,\,\,\forall\,f\in F$;

  (2) $G\subset_{\varepsilon}B|_Z\in \mathcal{C}$.
\end{thrm}
\begin{proof}
  Set $n=L(B)$, choose $\varepsilon',\eta,\eta_1$ as in Lemma \ref{ptb2}, then there exists $\delta>0$
  such that for any $dist(y,y')<\delta$, we have the following:
  $$
  \|\phi_y(h)-\phi_{y'}(h)\|<1,\quad\forall\,\,h\,\in H(\eta_1),
  $$
  $$
  \|\phi_y(h)-\phi_{y'}(h)\|<\frac{\varepsilon'}{8},\quad\forall\, h\in H(\eta)\cup \widetilde{H}(\eta)
  $$
  $$
  \|g(y)-g(y')\|<\varepsilon,\quad\forall\,\,g\in G.
  $$
  Apply Lemma \ref{jihu}, we can obtain a closed subset $Z$ and a surjective map $\rho:Y\rightarrow Z$, such that $G\subset_{\varepsilon}B|_Z\in \mathcal{C}$.

  We will define an injective homomorphism $\psi:A\rightarrow B|_Z$ as follows.

  Recall the construction of $\hat{Y}$ and $P(Z)$ in \ref{deff}. Let $P(Z)=\{z_1,z_2,\cdots\}$ be the points corresponding to the finite points
  $\{y_1,y_2,\cdots\}=\hat{Y}$. Define
  $$
  %\psi(f)(z_k)=\psi(f)(\rho(y_k))=\phi(f)(y_k),\quad z_k\in\{z_1,z_2,\cdots\}.
  \psi_{z_k}(f)=\psi_{\rho(y_k)}(f)=\phi_{y_k}(f),\quad \forall f\in A,\,\, z_k\in\{z_1,z_2,\cdots\}.
  $$
  For each adjacent pair $\{(y_s,i),(y_t,i)\}$, if $(y_s,y_t)_i\cap Y$ has at most countably many points, then $(z_s,z_t)_i\cap Z=\varnothing$, we don't need to define $\psi$ on $(z_s,z_t)_i$, if $(y_s,y_t)_i\cap Y$ has uncountable many points, then we have $dist((y_s,i),(y_t,i))<\delta$ and $[z_s,z_t]_i\subset Z$, then by
  Lemma \ref{ptb2}, we can define $\psi$ on $[z_s,z_t]_i$ and
  $$
  \|\psi_z(f)-\phi_{(y_s,i)}(f)\|<\varepsilon,\,\,\forall f\in F,\,\forall\,z\in [z_s,z_t]_i.
  $$
   Apply the above procedure to all adjacent pairs in $\hat{Y}$, we can define $\psi$ on each $[z_s,z_t]_i\subset Z$ piece by piece, then we obtain $\psi$ on $Z\cap \coprod_{i=1}^{l}[0,1]_i$. For each $\theta_j\in Z\cap Sp(F_1)$, define $\psi_{\theta_j}(f)=\phi_{\theta_j}(f)$ for all $\theta_j\in Y\cap Sp(F_1)$. Then we have defined $\psi$ on $Z$ and $\psi$ satisfys property (1).

  To prove $\psi$ is injective, we only need to verify that
  $Sp\psi=\bigcup_{z\in Z}Sp\psi_z=Sp(A)$.
  The proof is similar to the corresponding part of \cite{Li1}.

  Write $A=\bigoplus_{k=1}^mA_k$ with all $A_k$ are minimal, then $Sp(A)=\coprod_{k=1}^m Sp(A_k)$. Define an index set $\Lambda\subset\{1,2,\cdots,m\}$ such that $A_k$ is a finite dimensional $\mathrm{C}^*$-algebra iff $k\in \Lambda$. For $k\in \Lambda$, $\phi|_{A_k}\neq 0$ means that $Sp(A_k)\subset Sp\phi$, by the definition of $\psi$, we have $\psi|_{A_k}\neq 0$, then $Sp(A_k)\subset Sp\psi$.

  Consider $\widetilde{A}=\widetilde{A}(\widetilde{F}_1,\widetilde{F}_2,
  \widetilde{\varphi}_0,\widetilde{\varphi}_1)=\bigoplus_{k\notin \Lambda}A_k$, we define two sets $Y',Y''\subset Y$, for each adjacent pair $\{(y_s,i),(y_t,i)\}$, if $(y_s,y_t)_i\cap Y$ has at most countably many points, let $(y_s,y_t)_i\cap Y\subset Y'$, if $(y_s,y_t)_i\cap Y$ has uncountable many points, let $[y_s,y_t]_i\cap Y\subset Y''$. Then  we have $Y'\cap Y''=\varnothing$ and $Y'\cup Y''=Y\cap \coprod_{i=1}^{l}[0,1]_i$, note that $Y'$ has at most countably many points.

  For any point $x_0\in \coprod_{i=1}^{l}(0,1)_i$ and
  $\overline{B_{\eta_1}(x_0)}=\{x\in Sp(\widetilde{A}):\,dist(x,x_0)\leq\eta_1\}$,
  $\overline{B_{\eta_1}(x_0)}\cap(\bigcup_{y\in Y'}Sp\phi_y)$ have at most countably many points. Follow the injectivity of $\phi$, we have
  $$
  \overline{B_{\eta_1}(x_0)}\subset Sp\phi=\bigcup_{y\in
  Y''}
  Sp\phi_y\cup \bigcup_{y\in Y'}Sp\phi_y\cup\bigcup_{y\in Y\cap Sp(\widetilde{ F}_1)}Sp\phi_y.
  $$
  Then the set $\bigcup_{y\in Y''}Sp\phi_y\cap\overline{B_{\eta_1}(x_0)}$ has uncountably many points, recall the definition of $Y''$,
  there is at least one adjacent pair $\{(y_s,i),(y_t,i)\}$ such that $[y_s,y_t]_i\cap Y$ has uncountably many points, then we have $\psi$  defined on $[z_s,z_t]_i\subset Z$.

  Choose $$x_1\in \bigcup_{y\in[y_s,y_t]_i\cap Y''}Sp\phi_y\cap \overline{B_{\eta_1}(x_0)},$$
  then there exists
  $x_2\in Sp\phi_{(y_s,i)}$ such that $dist(x_1,x_2)\leq2\eta_1$, we have
  $$
  dist(x_0,x_2)\leq dist(x_0,x_1)+dist(x_1,x_2)\leq3\eta_1<4\eta_1.
  $$

  By Lemma \ref{ptb2}, we will have
  $$
  x_0\in \overline{B_{4\eta_1}(x_2)}\subset \bigcup_{z\in[z_s,z_t]_i}Sp\psi_z
  $$
  This means that $\coprod_{i=1}^{l}(0,1)_i\subset Sp\psi$.

  Note that, if we choose $x_0$ such that
  $x_0\in \coprod_{i=1}^{l}(0,\eta_1)_i\cup(\eta_1,1)_i$,
  then we will have $0_i,1_i\in Sp\psi$ for all $i\in \{1,2,\cdots,l\}$, this means that $Sp(\widetilde{F}_1)\subset Sp\psi$.

  Now we have
  $$
  Sp\psi=\bigcup_{z\in Z}Sp\psi_z=Sp(\widetilde{A})\cup \coprod_{k\in \Lambda}Sp(A_k)=Sp(A).
  $$
  This ends the proof of the injectivity of $\psi$.
\end{proof}
\begin{remark}\label{nonmore}
  Theorem \ref{mainthrm} still holds if we let $\phi$ be non-unital, then the homomorphism $\psi$ will also be non-unital.
\end{remark}
\begin{ddef}
  \textbf{Proof of Theorem 3.1} \cite{Li1}. Let $\widetilde{A}_n=\phi_{n,\infty}(A_n)$, $n=1,2,\cdots$. Then we can write
  $A=\lim_{n\rightarrow\infty}(\widetilde{A}_n,\tilde{\phi}_{n,m})$, where the homomorphism $\widetilde{\phi}_{n,m}$ are
  induced by $\phi_{n,m}$, and they are injective.

  Let $\varepsilon_n=\frac{1}{2^n}$, $\{x_i\}_{i=1}^{\infty}$ be a dense subset of $A$. We will construct an injective inductive limit
  $B_1\rightarrow B_2\rightarrow \cdots$ as follows.

  Consider $G_1={x_1}\subset A$. There is an $\widetilde{A}_{i_1}$, and a finite subset $\tilde{G}_1\subset \widetilde{A}_{i_1}$ such that
  $G_1\subset_{\frac{\varepsilon_1}{2}}\widetilde{G}_{i_1}$.

  For $\widetilde{G}_1\subset\widetilde{A}_{i_1}$, apply Lemma \ref{jihu}, there exists a sub-algebra $B_1\subset\widetilde{A}_{i_1}$ such that
  $B_1\in \mathcal{C}$ and $\widetilde{G}_1\subset_{\frac{\varepsilon_1}{2}} \widetilde{B}_1$. This give us an injective homomorphism
  $B_1\hookrightarrow \widetilde{A}_{i_1}$. Let $\{b_{1j}\}_{j=1}^{\infty}$ be a dense subset of $B_1$. Set $\widetilde{F}_1=\{b_{11}\}\subset B_1$
  and $G_2=\{x_1,x_2\}\subset A$. There exist $\tilde{A}_{i_2}$, $i_2>i_1$ and a finite subset $\widetilde{G}_2\subset \widetilde{A}_{i_2}$ such that
  $G_2\subset_{\frac{\varepsilon_2}{2}}\widetilde{G}_2$. Apply Theorem \ref{mainthrm} and Remark \ref{nonmore} to $\widetilde{F}_1\subset B_1$, $\widetilde{G}_2\subset \widetilde{A}_{i_2}$, and
  the injective map $B_1\hookrightarrow \widetilde{A}_{i_1}\rightarrow \widetilde{A}_{i_2}$, there exist a sub-algebra $B_2\subset\widetilde{A}_{i_2}$ and
  an injective homomorphism $\psi_{1,2}:B_1\rightarrow B_2$ such that $\widetilde{G}_2\subset_{\frac{\varepsilon_2}{2}} \widetilde{B}_2$ and such that
  the diagram
  $$
  \begin{matrix}
    \widetilde{A}_{i_1} & \xrightarrow{\widetilde{\phi}_{i_1,i_2}} & \widetilde{A}_{i_2} \\
    \uparrow &  & \uparrow \\
    B_1 & \xrightarrow{\psi_{1,2}} & B_2
  \end{matrix}
  $$
  almost commutes on $\widetilde{F}_1$ to within $\varepsilon_1$. Let $\{b_{2j}\}_{j=1}^{\infty}$ be a dense subset of $B_2$.
  Choose
  $$
  \widetilde{F}_2=\{b_{21},b_{22}\}\cup\{\psi_{1,2}(b_{11}),\psi_{1,2}(b_{12})\},\quad G_3=\{x_2,x_2,x_3\}
  $$
  in the place of $\widetilde{F}_1$ and $G_2$ respectively, and repeat the above construction to obtain
  $\widetilde{A}_{i_3}$, $B_3\subset \widetilde{A}_{i_3}$ and an injective map $\psi_{2,3}:B_2\rightarrow B_3$
  (Using $\varepsilon_2$ and $\varepsilon_3$ in place of $\varepsilon_1$ and $\varepsilon_2$, respectively).

  In general, we can construct the diagram
  $$
 \begin{matrix}
\widetilde{A}_{i_1} & \xrightarrow{\widetilde{\phi}_{i_1,i_2}} & \widetilde{A}_{i_2} & \xrightarrow{\widetilde{\phi}_{i_2,i_3}} &
\widetilde{A}_{i_3} & \rightarrow\,\,\cdots & \widetilde{A}_{i_k} & \rightarrow\,\,\cdots
 \\
\uparrow &  & \uparrow &  & \uparrow &  & \uparrow &  \\
B_1 & \xrightarrow{\psi_{1,2}} & B_2 & \xrightarrow{\psi_{2,3}} & B_3 & \rightarrow\,\,\cdots
 & B_k & \rightarrow\,\,\cdots
\end{matrix}
$$
  with the following properties:

  (i) The homomorphism $\psi_{k,k+1}$ are injective;

  (ii) For each $k$, $G_k=\{x_1,x_2,\cdots,x_k\}\subset_{\varepsilon_k} \widetilde{\phi}_{i_k,\infty}(B_k)$, where
  $B_k$ is considered to be a sub-algebra of $\widetilde{A}_{i_k}$;

  (iii) The diagram
  $$
  \begin{matrix}
    \widetilde{A}_{i_k} & \xrightarrow{\widetilde{\phi}_{i_k,i_{k+1}}} & \widetilde{A}_{i_{k+1}} \\
    \uparrow &  & \uparrow \\
    B_k & \xrightarrow{\psi_{k,k+1}} & B_{k+1}
  \end{matrix}
  $$
  almost commutes on $\widetilde{F}_k=\{b_{ij};\, 1\leq i\leq k,\,1\leq j\leq k\}$ to within $\varepsilon_k$,
  where $\{b_{ij}\}_{j=1}^{\infty}$ is a dense subset of $B_i$.

  Then by 2.3 and 2.4 of \cite{Ell2}, the above diagram defines a homomorphism from $B=\underrightarrow{lim}(B_n,\psi_{n,m})$ to
   $A=\underrightarrow{lim}(\widetilde{A}_n,\widetilde{\phi}_{n,m})$. It is routine to check that the homomorphism is in fact an isomorphism.
   This ends the proof.
\end{ddef}
\proof[Acknowledgements]
The author would like to thank Guihua Gong for helpful suggestions and discussions.


\begin{thebibliography}{10}

\bibitem{ALZ}
Q. An, Z. Liu and Y. Zhang, {\small\it On the Classification of Certain Real Rank Zero $\mathrm{C}^*$ Algebras}, to appear.

\bibitem{BBEK}
B. Blackadar, O. Bratteli, G. A. Elliott and A. Kumjian, {\small\it Reduction of real rank in inductive limits of $C^*$-algebras}.
 Math. Ann. {\bf 292}(1992), 111-126.

\bibitem{Ell1}
G. A. Elliott, {\small\it A classification of certain simple $\mathrm{C}^*$-algebras}. In: Quantum and Non-Commutative Analysis, (H. Araki et al. eds.) Kluwer, Dordrecht, 1993, 373-385.

\bibitem{Ell2}
G. A. Elliott, {\small\it On the classification of $\mathrm{C}^*$-algebras of real rank zero}.  J. Reine. Angew. Math. {\bf 443}(1993), 179--219.

\bibitem{Ell3}
G. A. Elliott, {\small\it A classification of certain simple $\mathrm{C}^*$-algebras, II}.  J. Ramanujan Math. Soc. {\bf 12}(1997), 97-134

\bibitem{EG}
G. A. Elliott and G. Gong, {\small\it On the classification of $C^*$-algebras of real rank zero, II}, Ann. of Math. {\bf 144} (1996), 497-610.

\bibitem{EGL}
G. A. Elliott, G. Gong, and L. Li, {\small\it Injectivity of the connecting maps in $AH$ inductive systems.}
      Canad. Math. Bull. {\bf 48}(2005), 50-68.

\bibitem{ET}
G. A. Elliott and K. Thomsen, {\small\it The state space of the $K_0$-group of a simple separable $\mathrm{C}^*$-algebra}, Geom. Funct. Anal.
{\bf 4} (5) (1994), 522--538.

\bibitem{G}
G. Gong, {\small\it On the classification of simple inductive limit $\mathrm{C}^*$-algebras, I. The reduction theorem.} Doc. Math. {\bf 7}(2002), 255-461.

\bibitem{GL}
G. Gong, and H. Lin, {\small\it On classification of simple non-unital amenable $\mathrm{C}^*$ algebras, II}. arXiv:1702.01073.

\bibitem{GLN}
G. Gong, H. Lin, and Z. Niu,  {\small\it Classification of finite simple amenable $\mathcal{Z}$-stable $C^*$-algebras}. arXiv:1501.00135v6.

\bibitem{Li1}
L. Li, {\small\it Classification of simple $\mathrm{C}^*$ algebras inductive limits of matrix algebras over trees}.
  Mem. Amer. Math. Soc. {\bf 127}(605)1996.

\bibitem{Li2}
L. Li, {\small\it Simple inductive limit $\mathrm{C}^*$ algebras: Spectra and approximation by interval algebras}. J. Reine Angew. Math.,
      {\bf 507}(1999), 57-79

\bibitem{Li3}
L. Li, {\small\it Classification of simple $\mathrm{C}^*$ algebras: inductive limit of matrix algebras over 1-dimensional spaces}.
       J. Funct. Anal., {\bf 192}(2002), 1-51.

\bibitem{Liu}
Z. Liu, {\small\it A decomposition theorem for real rank zero inductive limits of 1-dimensional non-commutative CW complexes.} arXiv:1709.03684v1.


\bibitem{Su}
H.Su, {\small\it On the classification of $C^*$-algebras of real rank zero: Inductive limits of matrix algebras over non-Hausdorff graphs}. Mem. Amer. Math. Soc. {\bf 114}(547) (1995).
\end{thebibliography}
\end{document}